\documentclass{amsart}
\usepackage{color}
\usepackage{psfrag,graphicx,tikz}

\newcommand{\R}{\ensuremath{\mathbb{R}}}

\renewcommand{\div}{\mathop{\rm Div}\nolimits}
\newcommand{\Mod}{\mathop{\rm Mod}}
\newcommand{\CC}{C}
\newcommand{\V}{l}
\newcommand{\sdot}{\! \cdot \!}

\newcommand{\Ball}[2]{B_{#1} \!\left(#2\right)}
\newcommand{\ball}[1]{B_{#1} }

\theoremstyle{plain}
\newtheorem{theorem}{Theorem}[section]

\newtheorem{lemma}[theorem]{Lemma}

\numberwithin{equation}{section}

\theoremstyle{definition}

\newtheorem{remark}[theorem]{Remark}

\begin{document}

\title{Filling loops at infinity in the mapping class group}
\author[A. Abrams]{Aaron Abrams}
\address{
Aaron Abrams\\
Mathematics Department\\
Robinson Hall\\ 
Washington and Lee University\\
Lexington VA 24450}
\email{abrams.aaron@gmail.com}
\author[N. Brady]{Noel Brady}
\address{
Noel Brady\\
Department of Mathematics\\
University of Oklahoma\\
601 Elm Ave\\
Norman, OK 73019}
\email{nbrady@math.ou.edu}
\author[P. Dani]{Pallavi Dani}
\address{
Pallavi Dani\\
Department of Mathematics\\
Louisiana State University\\
Baton Rouge, LA 70803-4918}
\email{pdani@math.lsu.edu}
\author[M. Duchin]{Moon Duchin}
\address{
Moon Duchin\\
Department of Mathematics\\
Tufts University\\
503 Boston Ave.\\
Medford, MA 02155}
\email{moon.duchin@tufts.edu}
\author[R. Young]{Robert Young}
\address{
Robert Young\\
Department of Mathematics\\
University of Toronto\\
40 St.\ George St., Room 6290\\
Toronto, ON  M5S 2E4\\
Canada}
\email{ryoung@math.toronto.edu}
\date{\today}
\thanks{This work was supported by a {\sf SQuaRE} grant from the
  American Institute of Mathematics.  The second author and 
  the fourth author are partially supported by NSF grants DMS-0906962 and 
  DMS-0906086, respectively.
  The fifth author would like to
  thank New York University for its hospitality during the preparation
of this paper.  We thank the referee for helpful comments.}

\begin{abstract}
We study the Dehn function at infinity in the mapping class group,
finding a polynomial upper bound of degree four.  This is the same upper
bound that holds for arbitrary right-angled Artin groups.
\end{abstract}

\maketitle


Dehn functions quantify simple connectivity.  That is,
in a simply-connected space, every closed curve is the boundary of some disk; the Dehn function measures the area required to fill the curves of a given length.
The growth of the Dehn function is invariant under quasi-isometry, so one can define the Dehn function not just for spaces, but also for groups.  The Dehn function is not the only group invariant based on a filling problem; for example, one can also define the {\em Dehn function at 
infinity}, which is a quasi-isometry invariant that measures the difficulty of filling closed curves with disks that avoid a large ball.
The Dehn function at 
infinity is a special case ($k=1$) of the  {\em higher divergence functions} $\div^k$ that were defined 
for groups in \cite{abddy} and serve to quantify the connectivity at infinity.
In that paper we survey some results using the growth rates of $\div^k$ to detect geometric features of
groups and spaces.

The mapping class group of a surface has quadratic Dehn function because it is automatic, and an automatic structure provides a combing which can be used to shrink a curve to a point using no more area than is needed in a Euclidean space
(see~\cite{mosher, wordproc}).  In this note we study the Dehn function at infinity:
if we impose the additional condition 
that the filling of a loop avoid a large ball, must its area be much worse than quadratic?  
In~\cite[Theorem 6.1]{abddy} we addressed this question and its higher-dimensional analogs in the 
case of right-angled Artin groups (RAAGs), and we showed that loops can
be filled at infinity using area at most polynomial of degree four.   

Here we show that the same result holds in mapping class groups of surfaces of genus
$g\ge 5$, contributing to the growing literature comparing mapping class groups to RAAGs.
We use two key features of these mapping class groups:  first, they have 
presentations with short relators (due to Gervais~\cite{gervais}), and second, all abelian
subgroups are undistorted \cite{flm}.


\section{Notation}
In any space $X$ we denote by $\Ball rx$  the ball of radius $r$ centered at $x$.  
We will use $x_0$ to denote a basepoint and we often write $\ball r$ for $\Ball r{x_0}$.
Any object in $X$ that is disjoint from $\ball r$ is called \emph{$r$-avoidant} (or often simply
\emph{avoidant}).

As usual, for two functions $f,g:\R\to\R$ we write $f\preceq g$ if there is a constant $A>0$
such that for all $t\geq 0$,
$$f(t) \le Ag(At+A)+At+A.$$

\begin{remark}\label{trivialEuclidean}
Note that in Euclidean space $\R^d$, any two points on the sphere of radius $r$
can be joined by an $r$-avoidant path of length at most $\pi r$.
Also, it is an exercise that there exists a constant $c>0$ such that
any $r$-avoidant loop of length $l$ in any $\R^d$ ($d\ge3$) can be filled
with an $r$-avoidant disk of area at most $cl^2$.
\end{remark}


\section{The Gervais presentation} \label{sec:MCG}
For a topological surface $S=S_{g,b}$ (where $g$ is the genus and $b$ is the 
number of punctures/boundary components), we write $\Mod(S)$ for its group 
of orientation-preserving diffeomorphisms up to isotopy, or {\em mapping class group}.

Our strategy for bounding the area of an efficient $r$-avoidant filling in $\Mod(S)$ is based 
on the ideas developed in \cite{abddy} for RAAGs.  
We begin with an efficient but presumably non-avoidant filling and alter it, ``pushing" each 
original 2-cell to an avoidant 2-cell (i.e., replacing the former with the latter). 
These new cells are then patched together 
(still avoidantly) using commuting relations to form the new filling.  Careful control 
of the pushing process allows us to bound the number of $2$-cells in the new filling in terms of
the number of $2$-cells in the original filling.

In order to do this, we present $\Mod(S)$ as a quotient of a RAAG
whose generators are Dehn twists (which commute if the corresponding
curves are disjoint).  In a 2-complex for such a presentation, all cells are either
squares coming from the RAAG or among finitely many types of other cells coming from the additional
relators of the mapping class group.

Squares coming from commutation relations can be replaced by avoidant squares
in a straightforward manner:  we push them out radially
by post-multiplying with a high power of one of the two commuting letters.   
The effect of this is to translate along 
a standard ray in the $2$-complex.
For the other types of $2$-cells, we will be able to carry out a similar pushing operation
if we can find a common commuter for all of the letters in the corresponding relator.
That is, if $\sigma$ is a $2$-cell with boundary labelled by the word $w$, and $h$ is a generator which commutes with all letters in $w$, then 
post-multiplying by $h^R$ results in 
a translated copy of $\sigma$ that is far from $x_0$.
To employ this strategy, we would like a presentation in which every 
generator is a Dehn twist, and every 
relator $w$ has 
``small support'' in the following sense:  the curves corresponding to the letters appearing in $w$ are collectively supported on a subsurface $F$ such that some other generator has support disjoint from $F$. This generator therefore commutes with every letter in $w$, and can be used to push the corresponding cell as above.

\begin{figure}[ht]
\includegraphics[width=2.4in]{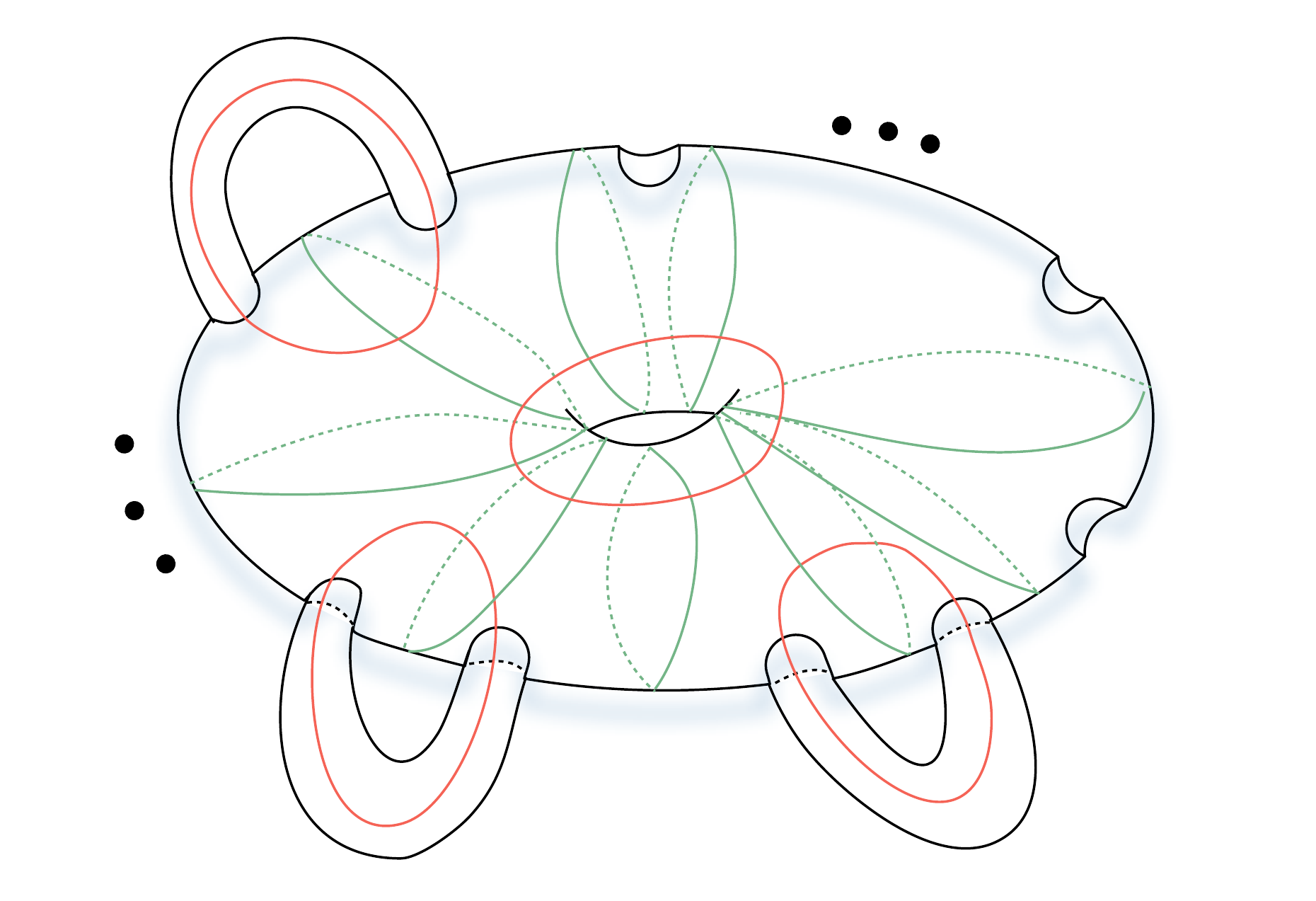}
\includegraphics[width=2.2in]{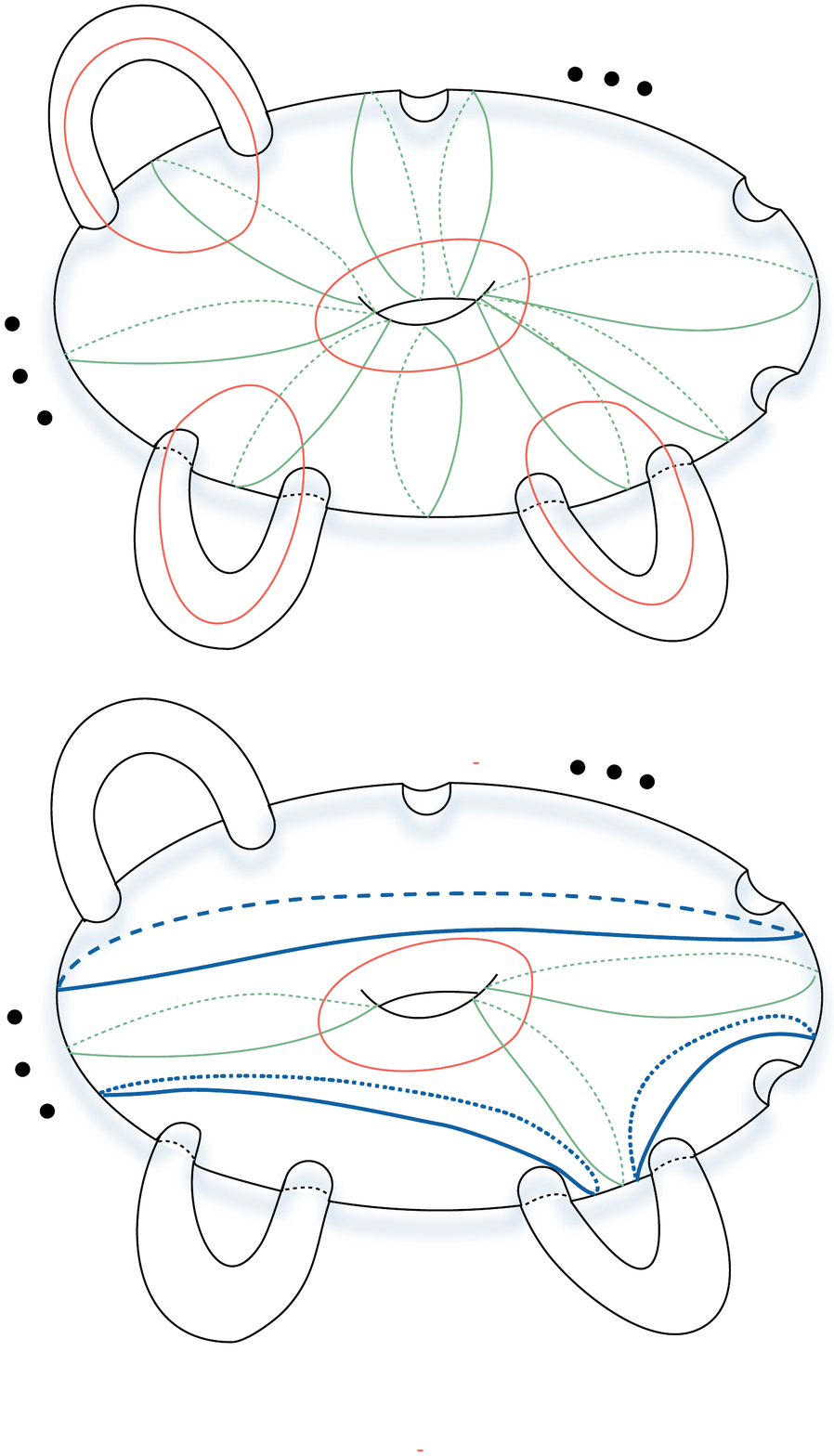}
\caption{A diagram of the Gervais curves on $S_{g,b}$.
On the left are $\alpha_1,\ldots,\alpha_{2g+b-2}$ (appearing as meridians of the central torus)
and $\beta_1,\ldots,\beta_g$ (with $\beta_1$ as the longitude of the central torus).
For every pair  $\alpha_i,\alpha_j$, there is a corresponding curve $\gamma_{ij}$;
the figure on the right shows three of the $\gamma_{ij}$.
The  star relations are formed by twists around 
seven curves in the configuration depicted on the right; the central curve $\beta_1$ is always
used.  Note that each such relation is supported on a three-times punctured torus.\label{g-curves}}
\end{figure}

The Gervais presentation of the mapping class group (see~\cite{gervais}) fits the bill:
this is a finite presentation in which every relator is supported on a small
subsurface (at most a three-times-punctured torus).  
Compared to better-known presentations (such as with Humphries or Lickorish generators), 
this will have the advantages provided by common commuters
to offset the disadvantage of having a far larger number of generators.

All the Gervais generators are Dehn twists 
supported on the collection of curves shown in 
Figure~\ref{g-curves}.  We call these the {\em Gervais curves}.  
The relations are of three kinds:  commuting, braid,
and so-called {\em star} relations. 
Commuting relations arise from twists around disjoint curves.  
The braid relations arise whenever two curves intersect once;
they have the form $ABA=BAB$.  The star relations arise when a collection of seven curves is 
in a particular topological configuration (see the right side of Figure~\ref{g-curves}); these have the form $(ABCD)^3=XYZ$.  
If two of the twisting curves for $A,B,C,D$ are isotopic (say $A$ and $B$), then there is a 
corresponding degenerate star relation (in this case $(A^2CD)^3=XY$).

For this presentation of $\Mod(S)$,
let $X$ be the universal cover of the presentation 2-complex, so that its 1-skeleton
is the Cayley graph.  Then all of the 2-cells are squares, hexagons, 14-gons, and 15-gons corresponding to 
the relators described above.  

The possible directions to push 2-cells, as well as the possible commuting 
relations used to patch avoidant $2$-cells together, are determined by a particular abelian subgroup of 
$\Mod(S)$ generated by Dehn twists around the set of curves described in  
the following lemma.

\begin{lemma}[Common commuters in the Gervais
  presentation]\label{lem:commuters}  Let $S$ be an orientable surface of genus $g\ge 5$
  and any number $b\ge 0$ of punctures.
There is a set $\mathcal{H}$ of $g$ mutually disjoint Gervais curves, whose associated Dehn twists 
generate a subgroup $H\le \Mod(S)$, with the following properties: 

\begin{enumerate}
\item For every relation in the Gervais presentation, there exists an element of 
$\mathcal{H}$ whose Dehn twist commutes with every letter appearing in the relation.
\item Any Gervais curve intersects at most two curves from $\mathcal{H}$.
\end{enumerate}
\end{lemma}

\begin{proof} 
Gervais gives his presentation in terms of three types of curves:  
$\alpha$-curves (separating out the topology),
$\beta$-curves (mutually disjoint curves, one around each handle, dual to some of the $\alpha$-curves), and 
$\gamma$-curves (derived from pairs of $\alpha$-curves).  (See Figure~\ref{g-curves}.)
Each star relation is supported on a three-times-punctured torus, and the relation 
involves three $\alpha$-curves, one $\beta$-curve, and three $\gamma$-curves.  
One easily verifies that the maximum number of other $\beta$-curves intersecting 
any of these support curves is three, if dual to the $\alpha$-curves.  Thus if 
there are at least five $\beta$-curves in total, one of them must be disjoint from the 
support curves.  In this case, the support of any commuting or braid relation clearly 
also misses some $\beta$-curve.  
The total number of $\beta$-curves is $g$, the genus.  

Each $\alpha$-curve and each $\gamma$-curve intersects at most two $\beta$-curves, by inspection.
Now let $\mathcal{H}$ be the set of all $\beta$-curves.
\end{proof}


\section{Filling loops at infinity}

Let $h_1, \dots h_g$ denote the Dehn twists about the $\beta$-curves,
and let $H$ denote the subgroup $\langle h_1, \dots h_g \rangle$.  The
group $H$ is abelian, since the $\beta$ curves are disjoint.  Abelian
subgroups of $\Mod(S)$ are undistorted (\cite{flm}).  That is, if $d$
is the word metric on $\Mod(S)$ with respect to the Gervais
presentation, and $d_H$ is the word metric on $H$, then there is a
$\CC>1$ such that
\begin{equation*}\label{eq:distortion}
  d(x,y) \le d_H(x,y) \le \CC  d(x,y) + \CC. \tag{{\sc und}}
\end{equation*}

This means that the intersection of $\ball{r}$ and any coset of $H$ is
small:
\begin{lemma}
  Let $H'\subset H$ be generated by a subset of the $\beta$-curves.
  Let $v\in \Mod(S)$ and let $y\in v\cdot H'$ be a point in $v\cdot H'$ such that
  $d(x_0,y)=d(x_0,v\cdot H')$.  If $\Ball ry$ denotes the $r$-ball in $v\cdot H'$
  centered at $y$ and $r\ge 1$, then 
  $$\Ball{r-d(x_0,v\cdot H')}y \subset \ball{r}\cap v\cdot H'\subset \Ball{3\CC r}y.$$
\end{lemma}
\begin{proof}
  The first inclusion follows from the triangle inequality and the
  fact that $d \le d_H$.  For the second, suppose that $z\in \ball{r}\cap
  v\cdot H'$.  Then by the definition of $y$, we
  have $d(x_0,y)\le r$, so $d(z,y)\le 2r$ and
  $$d_{v\cdot H'}(z,y)\le 2 \CC r  + \CC\le 3 \CC r. \qedhere$$
\end{proof}
So we can construct avoidant curves and disks in $X$ from avoidant
curves and disks in $H'$:

\begin{lemma}[Avoidance in cosets]\label{lemma:relative avoidance}
  Suppose that $H'\subset H$ is generated by a subset of the $\beta$-curves and that
  $H'$ has rank at least 3.  There is a constant $D>0$ depending only on $S$
  such that for any $v\in \Mod(S)$ and any $r\ge D$:
  \begin{enumerate}
  \item Let $x_1,x_2\in v\cdot H'$ be
    $r$-avoidant.  There is an $\frac rD$-avoidant path in $v\cdot
    {H'}$ from $x_1$ to $x_2$ which has length at most $D \sdot   d_{H'}(x_1,x_2)$.
  \item Let $\gamma$ be an $r$-avoidant curve in $v\cdot H'$ of length
    $l$.  There
    is an $\frac rD$-avoidant disk $f:D^2\to v\cdot H'$ which fills
    $\gamma$ and has area at most $D l^2$.
  \end{enumerate}
\end{lemma}
\begin{proof}
  Let $D=5\CC$.  If $d(x_0,v\cdot H')> \frac rD$, the statements are trivial,
  since the $\frac rD$-ball doesn't intersect $v\cdot H'$.  

  Otherwise, by the lemma above, there is a $y\in v\cdot H'$ such that
  $$\Ball{\frac {4r}5 }y  \subset \ball{r}\cap v\cdot H'$$
  and 
  $$\ball{\frac rD} \cap v\cdot H' \subset \Ball{\frac {3r}5 }y$$
  Since $x_1$, $x_2$, and $\gamma$ are $r$-avoidant, they are outside
  $\Ball{\frac {4r}5 }y$. 
  Since $H'$ has rank at least 3, by
  Remark~\ref{trivialEuclidean} there is a
  curve from $x_1$ to $x_2$ as well as a disk filling $\gamma$ which both avoid
  $\Ball{\frac {3r}5 }y$; consequently, this curve and disk are
  $\frac rD$-avoidant in $\Mod(S)$.
\end{proof}

Thus we can construct avoidant fillings of loops that live in flat cosets; 
we will use these to build avoidant fillings of arbitrary loops.
 
\begin{theorem}[Filling loops at infinity in the mapping class group]\label{mcgfill}
Suppose $S$ has genus at least $5$ and any number of punctures, and let $X$
  be the Cayley 2-complex of $\Mod(S)$.  There is a constant $c>0$ such that for
  any $r$, any $r$-avoidant loop of length $\V$ has an $\frac rc$-avoidant
  filling of area $\le c r^2 \V^2$.
\end{theorem}

\begin{proof}
We start with an $r$-avoidant loop of length $\V$ in $X$.  Since the Dehn function is 
quadratic,
there exists a (not necessarily avoidant) filling $\Delta$ with area $\preceq \V^2$.  
We use $\Delta$ as a combinatorial model for an avoidant filling of the same loop.  The new filling is obtained by making the 
following replacements, which are depicted in Figures~\ref{fig:mcg-avoidant1}-\ref{fig:mcg-avoidant2} and 
are described more precisely below.

\begin{enumerate}
\item[{\em Step 1}] Each $2$-cell of $\Delta$ is replaced by (``pushed
  to") an avoidant copy of itself. 

\item[{\em Step 2}] Each edge  of $\Delta$ is replaced by 
a (possibly degenerate)
avoidant strip of squares of length $\preceq r$.  
An edge belonging to two $2$-cells is replaced by a strip connecting the two pushed copies of the cells.  An edge belonging to a 
single $2$-cell is necessarily part of the boundary loop, and is extended to a strip connecting the edge to the pushed copy of the 
cell.  

\item[{\em Step 3}] The result of the previous 
steps is topologically a punctured disk, with one boundary component equal to the original loop
and an additional boundary component corresponding to each vertex in 
$\Delta$.  Each boundary component of the latter kind is filled by an avoidant disk in an appropriate flat.
\end{enumerate}

\begin{figure}
\centering
\def\svgwidth{4in} 
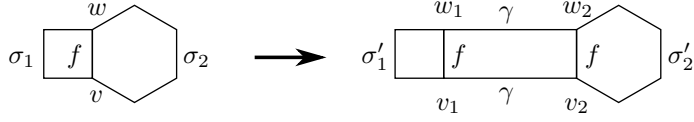 
\caption{\label{fig:mcg-avoidant1} The cells $\sigma_i'$ are obtained by pushing the
    $\sigma_i$ out of the ball of radius $r$.  The path $\gamma$ is
    $r/D$-avoidant, and the letters in the corresponding word all
    commute with $f$. }
\end{figure}


\begin{figure}
\centering
\def\svgwidth{4in} 
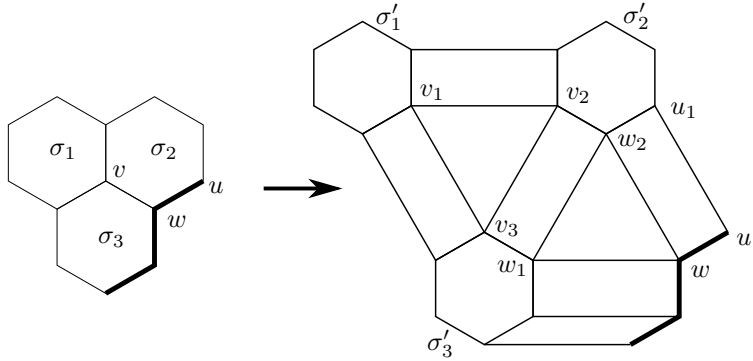 
\caption{\label{fig:mcg-avoidant2}The edges in bold in this figure are part of the boundary
    loop of $\Delta$.  Each strip is $(\frac rD-1)$-avoidant and has length
    $\preceq r$.}
\end{figure}


{\em Step 1:  Pushing $2$-cells.}  We replace each 2-cell $\sigma$ with an avoidant cell
$\sigma'$.  If $\sigma$ is
already $r$-avoidant, we let $\sigma'=\sigma$.  Otherwise, $\sigma$ is
partially contained in the ball of radius $r$.  It corresponds to a
relation in the Gervais presentation, and we choose a common commuter
$h_\sigma$ for the generators in this relation, whose existence is guaranteed by
Lemma~\ref{lem:commuters}.

Let $R=(2r+30)\CC+\CC$.  If the vertices of $\sigma$ are $v_1, \dots
v_k$, then $v_1h_\sigma^{R}, \dots, v_kh_\sigma^{R}$ are the vertices
of a copy of $\sigma$ (i.e., an isometric $2$-cell), 
since $h_\sigma$ commutes with the edge labels
of $\sigma$.  Denote this copy by $\sigma'$.  Since $\sigma$ is
partially contained in the ball of radius $r$ and relators have length at most $15$, 
$\sigma'$ is entirely outside the ball of radius $r+15$.  Thus it is $r$-avoidant by
undistortedness~\eqref{eq:distortion}.

{\em Step 2: Connecting pushed cells with strips.}  Consider an edge in $\Delta$ with
vertices $v$ and $w$, labeled by a Gervais generator $f$.  Let
$H_f\subset H$ be the subgroup of $H$ generated by the generators of
$H$ which commute with $f$; this has rank at least $g-2$, where $g$ is
the genus of $S$, by Lemma \ref{lem:commuters}(2).  

First, we find vertices $v_1$ and $v_2$ corresponding to $v$ in the
pushed filling.  If the edge is shared by two $2$-cells $\sigma_1$ and
$\sigma_2$, let $v_1,v_2$ be the vertices of $\sigma_1'$ and
$\sigma_2'$ which correspond to $v$.  Otherwise, the edge belongs to the  boundary
of $\Delta$.  If it is adjacent to a $2$-cell $\sigma$, let $v_1=v$
and let $v_2$ be the vertex in $\sigma'$ corresponding to $v$.
Otherwise, the edge is used twice in the boundary of $\Delta$, and  we can let
$v_1=v_2=v$.  In any case, $v_1$ and $v_2$ are $r$-avoidant and are
both contained in $v\cdot H_f$, and $d_{H_f}(v_1,v_2)\preceq r$.  By
Lemma~\ref{lemma:relative avoidance}, there is a $\frac rD$-avoidant path
$\gamma$ in $v\cdot H_f$ which connects $v_1$ to $v_2$; we can
interpret this as a word representing $v_1^{-1}v_2$ whose letters all
commute with $f$.  Then there is a strip built out of squares (that
is, commuting relations) whose boundary label is the commutator $[f,
\gamma]$; this strip is $(\frac rD-1)$-avoidant and has length $\preceq r$.

{\em Step 3: Filling in the holes.}  The partial filling constructed above has
one boundary component for each vertex of $\Delta$ which is
sufficiently close (within distance $r+15$) from the basepoint.  Each boundary
component is a polygonal loop whose sides are paths $\gamma$ belonging to strips from the previous step 
(these appear as triangles in Figure~\ref{fig:mcg-avoidant2}).  The number of sides
of the polygon associated to $v$ is the number of edges incident to
$v$ in $\Delta$.  Each vertex is $r$-avoidant, and each side is an
$(\frac rD-1)$-avoidant curve of length $\preceq r$.  Indeed, each vertex is
distance at most $R$ away from $v$, so any two vertices are distance
$\le 2R$ apart.  The entire polygon is contained in the coset $v\cdot
H$.

To fill these polygonal loops, we first subdivide each into triangular
loops by adding additional $\frac rD$-avoidant curves in $v\cdot H$ between
the vertices; these exist by Lemma \ref{lemma:relative avoidance}.
The resulting triangular loops are $(\frac rD-1)$-avoidant and have length
$\preceq r$.  Again by Lemma~\ref{lemma:relative avoidance}, they can
be filled by $(\frac r{D^2}-1)$-avoidant disks in $v\cdot H$ of area $\preceq
r^2$.  Let $\rho=\frac 1{2D^2}$, so that when $r$ is sufficiently large,
$\frac r{D^2}-1\ge \rho r$.

The union of the pushed cells, strips and filled triangles above is a
$\rho r$-avoidant filling of the 
boundary loop;
it remains to estimate the area of this filling.  Since the area (the number of $2$-cells) 
of $\Delta$ is $\preceq \V^2$, the number of vertices and edges in $\Delta$ is $\preceq \V^2$ 
as well.  
Each strip introduced in this construction has length (and area)
$\preceq r$.  
Because the triangular loops lie in cosets, 
there is a constant $M$ such that the area of any of the triangle fillings above is at most $Mr^2$.  
Moreover, the number of triangles is certainly bounded above by twice the number of edges in 
$\Delta$.  So the total area of the new filling is $\preceq \V^2 +
\V^2 (2r) + \V^2 (Mr^2)\preceq r^2\V^2$, as desired.
\end{proof}

In the language of \cite{abddy}, which primarily considers the case
that $l\sim r$, this implies that $\div^1(\Mod(S))\preceq r^4$.  By considering the lower bound from  Euclidean filling area, 
we have $r^2 \preceq \div^1(\Mod(S))$.  
In fact there is some evidence that the true answer is $\div^1(\Mod(S))\sim r^3$;  
good candidates for hard-to-fill loops are found in the $2$-flats generated by one Dehn twist
and one other mapping class that is pseudo-Anosov on the complementary subsurface.  
In general, to approach higher divergence
in mapping class groups, it is natural to focus on flats generated by many partial pseudo-Anosovs, because
axes of pseudo-Anosovs are strongly contracting \cite[Theorem 4.2]{duchin-rafi}, so it requires high volume to project down 
from far away to a 
given size in a grid made up by such axes.

\bibliographystyle{amsplain}
\bibliography{divk}

\providecommand{\bysame}{\leavevmode\hbox to3em{\hrulefill}\thinspace}
\providecommand{\MR}{\relax\ifhmode\unskip\space\fi MR }
\providecommand{\MRhref}[2]{%
  \href{http://www.ams.org/mathscinet-getitem?mr=#1}{#2}
}
\providecommand{\href}[2]{#2}
\begin{thebibliography}{1}

\bibitem{duchin-rafi}
Moon Duchin and Kasra Rafi, \emph{Divergence of geodesics in {T}eichm\"uller
  space and the mapping class group}, Geom. Funct. Anal. \textbf{19} (2009),
  no.~3, 722--742.

\bibitem{wordproc}
D.~B.~A. Epstein, J.~W. Cannon, D.~F. Holt, S.~V.~F. Levy, M.~S. Paterson, and
  W.~P. Thurston, \emph{Word processing in groups}, Jones and Bartlett
  Publishers, Boston, MA, 1992. \MR{MR1161694 (93i:20036)}

\bibitem{flm}
Benson Farb, Alexander Lubotzky, and Yair Minsky, \emph{Rank-1 phenomena for
  mapping class groups}, Duke Math. J. \textbf{106} (2001), no.~3, 581--597.
  \MR{MR1813237 (2001k:20076)}

\bibitem{gervais}
Sylvain Gervais, \emph{A finite presentation of the mapping class group of a
  punctured surface}, Topology \textbf{40} (2001), no.~4, 703--725.
  \MR{MR1851559 (2002m:57025)}

\bibitem{mosher}
Lee Mosher, \emph{Mapping class groups are automatic}, Math. Res. Lett.
  \textbf{1} (1994), no.~2, 249--255. \MR{MR1266763 (95a:57023)}

\end{thebibliography}

\end{document}